\documentclass[12pt]{amsart}

\usepackage{amssymb, tikz}

\usepackage{hyperref}

\usepackage{amsmath}

\newtheorem{theorem}{Theorem}[section]

\newtheorem{lemma}[theorem]{Lemma}

\newtheorem{fact}[theorem]{Fact}

\newtheorem{proposition}[theorem]{Proposition}
\newtheorem{corollary}[theorem]{Corollary}
\newtheorem{claim}[theorem]{Claim}

\theoremstyle{definition}
\newtheorem{definition}[theorem]{Definition}

\newtheorem{remark}[theorem]{Remark}

\let \restr = \upharpoonright

\let \into = \longrightarrow

\let \sub = \subseteq
\let \elsub = \preccurlyeq

\let \av = \arrowvert

\let \a = \alpha
\let \b = \beta

\let \d = \delta

\let \l = \lambda
\let \k = \kappa
\let \m = \mu
\let \n = \nu
\let \p = \pi
\let \r = \rho
\let \t = \theta
\let \G = \Gamma
\let \D = \Delta

\let \s = \sigma
\let \x = \xi

\let \o = \omega

\let \al = \aleph

\let \mtcl = \mathcal
\let \mtbb = \mathbb

\DeclareMathOperator{\dom}{dom}
\DeclareMathOperator{\range}{range}

\DeclareMathOperator{\cf}{cf}
\DeclareMathOperator{\ot}{ot}
\DeclareMathOperator{\cl}{cl}
\DeclareMathOperator{\ssup}{sup}

\DeclareMathOperator{\Lim}{Lim}
\DeclareMathOperator{\Ord}{\textsf{Ord}}

\DeclareMathOperator{\CH}{\textsf{CH}}

\DeclareMathOperator{\Measuring}{\textsf{Measuring}}

\title{Corrigendum to ``Measuring club-sequences together with the continuum large''}

\author[D. Asper\'o]{David Asper\'o}

\address{David Asper\'o, School of Mathematics, University of East Anglia, Norwich NR4 7TJ, UK}

\email{d.aspero@uea.ac.uk}

\author[M.A. Mota]{Miguel Angel Mota}

\address{Miguel Angel Mota,  Departamento de Matem\'aticas,
ITAM,
01080, Mexico City, Mexico}

\email{motagaytan@gmail.com}

\date{}

\begin{document}


\maketitle
\pagestyle{myheadings}\markright{Corrigendum to ``Measuring club-sequences together with the continuum large''}

\begin{abstract}
 $\Measuring$ says that for e\-very sequence $(C_\delta)_{\delta<\omega_1}$ with each $C_\delta$ being a closed subset of $\delta$ there is a club $C\subseteq\omega_1$ such that for every $\delta\in C$,
a tail of $C\cap\delta$ is either contained in or disjoint from $C_\delta$. In our JSL paper  ``Measuring club-sequences together with the continuum large'' we claimed to prove the consistency of $\Measuring$ with $2^{\aleph_0}$ being arbitrarily large, thereby answering a question of Justin Moore. The proof in that paper was flawed. In the presented corrigendum we provide a correct proof of that result. The construction works over any model of $\textsc{ZFC} + \CH$
and can be described as the result of performing a finite-support forcing
%iteration
construction
with side conditions consisting of suitable symmetric systems
%graphs consisting of edges
 of models with markers.
\end{abstract}

\section{Mea culpa and introduction}

The present note is a corrigendum to our \cite{AM4}, published in the JSL in 2017. That same year, Tanmay Inamdar found out that our main proof in that paper had an error. Inamdar's argument does not exhibit an actual counterexample to the relevant claims in our proof in \cite{AM4}, but it does show, at the very least, that our proof in \cite{AM4} is incomplete as it does not establish those claims. In the present note we fix that proof.
Our main new ingredients are the following.
\begin{itemize}
\item The use of a suitable notion of `closed symmetric system of models with markers'.
\item The fact that at all stages $\b$ such that $\cf(\b)>\o$, only markers $\r<\b$ are allowed.
\item The use of a partial square sequence on ordinals of cofinality $\o_1$.
\end{itemize}

We will construct a certain sequence $(\mtcl P_\b\,:\,\b\leq\k)$ of partial orders (the forcing witnessing our main theorem will be $\mtcl P_\k$). This sequence is not obviously an iteration, in the sense of $\mtcl P_\a$ being a complete suborder  of $\mtcl P_\b$ for all $\a<\b$. On the other hand, we will have an $\o_2$-club $D$ of $\k$ such that $\mtcl P_\a$ is a complete suborder of $\mtcl P_\b$ for all $\a<\b$ in $D\cup\{\k\}$, so our construction will give rise after all to a forcing iteration of length $\k+1$ (as $\k$ will be regular).

%We have chosen to keep the relevant parts of the introduction from \cite{AM4}.

$\Measuring$ (see \cite{EMM}) is the following very strong form of the failure of Club Guessing at $\omega_1$.

\begin{definition} $\Measuring$ holds if and only if for every sequence $\vec C=(C_\delta\,:\,\delta\in \omega_1)$, if each $C_\delta$ is a closed subset of $\delta$ in the order topology, then there is a club $C\sub\omega_1$ such that for every $\delta\in C$ there is some $\alpha<\delta$ such that either \begin{itemize}

\item $(C \cap\delta)\setminus\alpha\subseteq C_\delta$, or
\item $(C\setminus\alpha)\cap C_\delta=\emptyset$.
\end{itemize}

\end{definition}

In the above definition, we will say that \emph{$C$ measures $\vec C$.}  $\Measuring$ is of course equivalent to its restriction to club-sequences $\vec C$ on $\omega_1$, i.e., to sequences of the form $\vec C=(C_\delta\,:\,\delta\in \Lim(\omega_1))$, where each $C_\delta$ is a club of  $\delta$ and where $\Lim(\omega_1)$ denotes the set of countable limit ordinals.

In this paper we prove that $\Measuring$ is consistent with $2^{\aleph_0}$ being arbitrarily large. This answers a question of Moore, who asked whether $\Measuring$ is consistent with $2^{\aleph_0}>\aleph_2$ (see also \cite{forcing-conseqs}).

Given an infinite cardinal $\k$ and a stationary set $S\sub\k$, we call  $\vec E=(E_\a\,:\,\a\in S)$ a \emph{partial square sequence on $S$} iff
\begin{itemize}
\item for each $\a\in S$, $E_\a$ is a club of $\a$ of order type $\cf(\a)$, and
\item for all $\a$, $\a'\in S$, if $\bar\a$ is a limit point of both $E_\a$ and $E_{\a'}$, then $E_\a\cap\bar\a=E_{\a'}\cap\bar\a$.
\end{itemize}

Also, given a regular cardinal $\m<\k$, we write $S^\k_\m$ to denote the set $\a<\k$ such that $\cf(\a)=\m$.

In our main theorem (Theorem \ref{mainthm}), we will need $\CH$ to hold, as well as the existence of a partial square sequence on $S^\k_{\o_1}$, for some  regular cardinal $\kappa>\o_2$ such that $2^{{<}\k}=\k$.
It is easy to obtain such partial square sequences using the following proposition.

\begin{proposition}\label{partial-sq-prop} Let $\m$ be an infinite regular cardinal, let $\k\geq\m$ be an infinite cardinal such that $\k<\aleph_\m$, and suppose $\Box_\l$ holds for every cardinal $\l$ such that $\m\leq\l\leq\k$. Then there is a partial square sequence on $S^{\k^+}_{\m}$.
\end{proposition}

\begin{proof}
We may obviously assume $\m\geq\o_1$. We construct a partial square sequence $(E_\a\,:\,\a\in\k^+\cap\cf(\m))$ by recursion on $\a$.

For every cardinal $\l$ such that $\m\leq\l\leq\k$, let $(C^\l_\a\,:\,\a\in\Lim(\l^+))$ be a fixed $\Box_\l$-sequence, which we may assume is such that $C^\l_\a\cap\l=\emptyset$ for all $\a>\l$.
To start with, we let $E_\a=C^\m_\a$ for all $\a\in\m^+\cap\cf(\m)$. Now let $\a\in\k^+\cap\cf(\m)$ be above $\m^+$. Let $\l$ be the unique cardinal such that $\l\leq\a<\l^+$, and note that in fact $\l<\a$ since $\k<\al_\m$. Let $\tau_\a=\ot(C^\l_\a)$ and let $\pi_\a:C^\l_\a\to \tau_\a$ be the collapsing function of $C^\l_\a$. We then let $E_\a=\pi_\a^{-1}``E_{\tau_\a}$, which is defined since $\cf(\tau_\a)=\m$ and $\tau_\a<\l<\a$. Since $E_\a$ is a club of $\a$ of order type $\m$, it suffices to show that if $\a'\in\a\cap\cf(\m)$ and $\bar\a$ is a limit point of both $E_\a$ and $E_{\a'}$, then $E_\a\cap\bar\a=E_{\a'}\cap\bar\a$.

Since $C_\a\cap\l=\emptyset$, $\a'$ is necessarily above $\l$. But then $E_{\a'}=\p_{\a'}^{-1}``E_{\tau_{\a'}}$, where $\tau_{\a'}=\ot(C^\l_{\a'})$ and $\pi_{\a'}:C^\l_{\a'}\to \tau_{\a'}$ is the collapsing function of $C^\l_{\a'}$. Since $\bar\a$ is a limit point of both $C^\l_\a$ and $C^\l_{\a'}$, we know that $C^\l_{\bar\a}=C^\l_\a\cap\bar\a=C^\l_{\a'}\cap\bar\a$. But then $\bar\tau=\pi_\a(\bar\a)=\pi_{\a'}(\bar\a)$ is a limit point of both $E_{\tau_\a}$ and $E_{\tau_{\a'}}$. This implies that $E_{\tau_\a}\cap\bar\tau=E_{\tau_{\a'}}\cap\bar\tau$, and therefore $E_\a\cap\bar\a=\pi^{-1}_\a``(E_{\tau_\a}\cap\bar\tau)=\pi^{-1}_{\a}``(E_{\tau_{\a'}}\cap\bar\tau)=\pi^{-1}_{\a'}``(E_{\tau_{\a'}}\cap\bar\tau)=E_{\a'}\cap\bar\a$.
\end{proof}

%We denote by $\mathfrak b(\omega_1)$ the minimal cardinality of an unbounded subset of $^{\omega_1}\omega_1$ mod.\ countable.

A partial order $\mtbb P$ is $\al_2$-Knaster if for every sequence $(q_\x\,:\,\x<\o_2)$ of $\mtbb P$-conditions there is a set $I\sub\o_2$ of cardinality $\al_2$ such that $q_\x$ and $q_{\x'}$ are compatible for all $\x$, $\x'\in I$. Every $\al_2$-Knaster partial order has of course the $\al_2$-chain condition.

Our main theorem is the following.

\begin{theorem}\label{mainthm} ($\CH$)
Let $\kappa>\o_2$ be a regular cardinal such that $2^{{<}\kappa}=\kappa$.
Suppose there is a partial square sequence on $S^\k_{\o_1}$.
There is then a partial order $\mtcl P$ with the following pro\-perties.
\begin{enumerate}
\item $\mtcl P$ is proper.
\item $\mtcl P$ is $\aleph_2$-Knaster.
\item $\mtcl P$ forces the following statements.
\begin{enumerate}
\item $\Measuring$
\item $2^{\aleph_0}=2^{\aleph_1}=\kappa$
%\item $\mathfrak b(\omega_1)=\cf(\kappa)$
\end{enumerate}
\end{enumerate}
\end{theorem}

The following is now an immediate corollary of Proposition \ref{partial-sq-prop} and Theorem \ref{mainthm}.

\begin{corollary}\label{main-corollary} ($\CH$) Let $\kappa$ be a regular cardinal such that $\o_2<\kappa<\al_{\o_1}$ and $2^{{<}\kappa}=\kappa$. Suppose $\Box_\l$ holds for every uncountable cardinal $\l$ such that $\l^+\leq\k$.
There is then a partial order $\mtcl Q$ with the following pro\-perties.
\begin{enumerate}
\item $\mtcl Q$ is proper.
\item $\mtcl Q$ is $\aleph_2$-Knaster.
\item $\mtcl Q$ forces the following statements.
\begin{enumerate}
\item $\Measuring$
\item $2^{\aleph_0}=2^{\aleph_1}=\kappa$
%\item $\mathfrak b(\omega_1)=\cf(\kappa)$
\end{enumerate}
\end{enumerate}
\end{corollary}

Much of the notation used in this paper follows the standards set forth in \cite{JECH} and \cite{KUNEN}. Other, less standard, pieces of notation will be introduced as needed.\newline

\textbf{Acknowledgment}: We thank Justin Moore for his valuable comments on earlier versions of this work. Special thanks are due to Tanmay Inamdar for detecting a serious flaw in \cite{AM4}, which compelled us to write the present corrigendum, and to Tadatoshi Miyamoto for also detecting some serious problems.

\section{The construction}\label{the_construction}

Let us assume $\CH$
and let $\k>\o_2$ be a regular cardinal with $2^{{<}\k}=\k$ and such that there is a partial square sequence on $S^\k_{\o_1}$. Let $$\vec E^1=(E^1_\d\,:\,\d\in S^\k_{\o_1})$$ be such a partial square sequence. Let also
 $$\vec E^0=(E^0_\d\,:\,\d\in S^\k_\o)$$ be a ladder system on $S^\k_\o$ (i.e., every $E^0_\d$ is a club of $\d$ of order type $\omega$) such that $\cf(\a)\leq\o$ for every $\d\in S^\k_\o$ and $\a\in E^0_\d$.
Let $\vec E=\vec E^0\cup\vec E^1$.

We fix some pieces of notation that will be used throughout the paper:
 If $N$ is a set such that $N\cap \omega_1\in \omega_1$, $\delta_N$ denotes $N\cap\omega_1$. $\delta_N$ is also called \emph{the height of $N$}.
 If $N_0$ and $N_1$ are $\in$-isomorphic models satisfying the Axiom of Extensionality, we denote the unique isomorphism $$\Psi_{N_0, N_1}:(N_0; \in)\into (N_1; \in)$$ by $\Psi_{N_0, N_1}$.

 %Given  a model $N\elsub H(\k)$ and $P\sub (\Ord\cap N)\times N$, we say that $(N, P)$ is a \emph{layered model}.

 If $T\sub H(\kappa)$ and $N$ is a  model, we will write $(N, T)$ as short-hand for $(N, T\cap N)$.

  We will use a variant of
 the following
 %variant of the
 notion of symmetric system from \cite{forcing-conseqs}.

 \begin{definition}\label{hom}
Let $T\subseteq H(\kappa)$ and let $\mathcal N$ be a collection of countable
%layered
 models.
We say that \emph{$\mathcal N$ is a
 $T$-symmetric system}
 %(of layered models)}
if and only if the following holds.

\begin{enumerate}

\item For every $N \in\mathcal N$, $(N; \in, T)$ is an elementary substructure of $(H(\kappa); \in, T)$.

\item Given $N_0$, $N_1\in\mathcal N$, if $\delta_{N_0}=\delta_{N_1}$, then there is a unique isomorphism $$\Psi_{N_0,  N_1}:(N_0; \in, T)\into (N_1; \in, T)$$

\noindent Furthermore, $\Psi_{N_0, N_1}$ is the identity on $N_0\cap N_1$.

\item For all $N_0$, $N_1$, $M\in \mathcal N$, if $M \in N_0$ and $\delta_{N_0}=\delta_{N_1}$, then $\Psi_{N_0, N_1}(M)\in\mathcal N$.

\item For all $N$, $M\in \mathcal N$, if $\delta_M<\delta_N$, then there is some $N'\in \mathcal N$ such that $\delta_{N'}=\delta_N$ and $M\in N'$.

\end{enumerate}

\end{definition}

Following a suggestion of Inamdar, we refer to condition (4) in the above definition as \emph{the shoulder axiom}.

Strictly speaking, the phrase `$T$-symmetric system' is ambiguous in general since $H(\kappa)$ may not be determined by $T$. However, in all practical cases $(\bigcup T)\cap\Ord = \kappa$, so $T$ does determine $H(\kappa)$ in these cases.

%\begin{remark} Suppose $T\sub H(\k)$ and $\mtcl N$ is a $T$-symmetric system of layered models.
%\begin{enumerate}
%\item If  $N$ is a model and $P_0$ and $P_1$ are such that $(N, P_0)$ and $(N, P_1)$ are both in $\mtcl N$, then $P_0=P_1$.
%\item $(M, \bar P)$, $(N, P)\in\mtcl N$, and $M\in N$, then also $P\in N$.
%\end{enumerate}
%\end{remark}

Given a set $X$, we let $\cl(X)=X\cup \overline{X\cap\Ord}$, where $\overline{X\cap\Ord}$ denotes the closure of $X\cap\Ord$ in the order topology. We refer to $\cl(X)$ as \emph{the closure of $X$}.

The variant of the notion of symmetric system we are referring to is the following.

\begin{definition}\label{hom-1}
Let $T\subseteq H(\kappa)$, let  $\r\leq\k$, and let $\mathcal N$ be a collection of countable
 models. We say that $\mtcl N$ is a \emph{closed $T$-$\r$-symmetric system}
if and only if the following holds.

\begin{enumerate}

\item For every $N \in\mathcal N$, $(N; \in, T)$ is an elementary substructure of $(H(\kappa); \in, T)$.

\item Given $N_0$, $N_1\in\mathcal N$, if $\delta_{N_0}=\delta_{N_1}$, then there is a unique isomorphism $$\Psi_{N_0,  N_1}:(N_0; \in, T)\into (N_1; \in, T)$$

\noindent Furthermore,  $\Psi_{N_0, N_1}$ extends to an isomorphism $$\Psi^\ast_{N_0, N_1}: (\cl(N_0); \in, N_0)\into (\cl(N_1); \in, N_1)$$ which is the identity on $\cl(N_0)\cap\cl(N_1)$.

\item For all $N_0$, $N_1$, $M\in \mathcal N$, if $M \in N_0$ and $\delta_{N_0}=\delta_{N_1}$, then $\Psi_{N_0, N_1}(M)\in\mathcal N$.

\item For all $N$, $M\in \mathcal N$, if $\delta_M<\delta_N$, then there are $N'$, $M'\in \mathcal N$ such that $\delta_{N'}=\delta_N$, $\d_{M'}=\d_M$, $M'\in N'$, and $\Psi_{M, M'}(\x)=\x$ for each $\x\in M\cap\r$.
\end{enumerate}\end{definition}

In contexts in which $T$ is not relevant, we will sometimes refer to closed $T$-$\r$-symmetric systems simply as `closed $\r$-symmetric systems'.

We will need the two following amalgamation lemmas. These lemmas are proved in \cite{forcing-conseqs} for symmetric systems of models, but almost literally the same proofs work for
closed
$\r$-symmetric systems of models.

%We will need the two following amalgamation lemmas. These lemmas are proved in \cite{forcing-conseqs} for symmetric systems but essentially the same proofs work for closed symmetric systems.

\begin{lemma}\label{amalg0-}
Let $T\sub H(\kappa)$, $\r\leq\k$, let $\mtcl N$ be a
closed
 $T$-$\r$-symmetric system, let $N\in \mtcl N$, and let $\mtcl M\in N$ be a
closed
$T$-$\r$-symmetric system. Assume that
\begin{enumerate}
\item $\mtcl N\cap N \subseteq \mtcl M$,
\item $M \cap \mtcl M = M \cap \mtcl N$ for each $M \in \mtcl N\cap N$, and
\item $\rho \in \cl(M)$ for every $M \in \mtcl M \cup \mtcl N$.
\end{enumerate}
Then, $$\mtcl W(\mtcl N, \mtcl M, N):=\mtcl N\cup\bigcup\{\Psi_{N, N'}``\mtcl M \,:\, N'\in \mtcl N,\,\delta_{N'}=\delta_N\}$$
is the $\subseteq$-minimal
closed
 $T$-$\r$-symmetric system $\mtcl W$ such that $\mtcl N\cup \mtcl M\subseteq \mtcl W$.
\end{lemma}

Given $T\sub H(\kappa)$, $\r\leq\k$, and closed $T$-$\r$-symmetric systems $\mtcl N_0$ and $\mtcl N_1$, we write $\mtcl N_0\cong_T \mtcl N_1$ iff $|\mtcl N_0|=|\mtcl N_1|=m$, for some $m<\o$, and there are enumerations $(N^0_i\,:\, i<m)$ and $(N^1_i\,:\, i<m)$ of $\mtcl N_0$ and $\mtcl N_1$, respectively, for which there is an isomorphism
$\Psi$ between $$(\bigcup(\dom(\mtcl N_0)); \in, T, N^0_i)_{i<m}$$ and $$(\bigcup(\dom(\mtcl N_1)); \in, T, N^1_i)_{i<m}$$ which is the identity on $$\bigcup(\dom(\mtcl N_0))\cap \bigcup(\dom(\mtcl N_1))$$

\begin{lemma}\label{amalg1-}  Let $T\sub H(\k)$ be a predicate, $\r\leq\k$, let $\mtcl N_0$ and $\mtcl N_1$ be
closed
$T$-$\r$-symmetric systems, and suppose $\mtcl N_0\cong_T\mtcl N_1$. Then $\mtcl N_0\cup\mtcl N_1$ is a
closed
$T$-$\r$-symmetric system. \end{lemma}

Let us fix a function $$\Phi:S^\k_{\o_2} \into H(\kappa)$$
such that $\Phi^{-1}(x)$ is stationary in $S^\k_{\o_2}$ for all $x\in H(\kappa)$. Note that $\Phi$ exists by $2^{{<}\kappa}=\kappa$.

We will define our sequence $(\mtcl P_\b\,:\,\b\leq\kappa)$ of forcing notions by recursion on $\beta$, together with a sequence $(\Phi_\b\,:\,\b<\kappa)$ of predicates of $H(\k)$. Given $\b<\kappa$, we will denote by $\mtcl T_\b$ the club of countable $N\sub H(\k)$ such that $(N; \in, \Phi_\b)\elsub (H(\k); \in, \Phi_\b)$. We will also let $\vec \Phi_\b=(\Phi_{\a+1}\,:\,\a<\b)$

To start with, we let $\Phi_0$ be the satisfaction predicate for the structure $(H(\k); \in, \Phi, \vec E)$.

We will call ordered pairs $(N, \r)$, where
\begin{itemize}
\item $N\elsub (H(\k); \in, \Phi, \vec E)$ is a countable model,
\item $\r<\k$, and
\item $N\in\mtcl T_{\a+1}$ for every $\a\in N\cap\r$
%\item $P(\a+1):=\{x\,:\, \la\a+1, x\ra\in P\}=\Phi_{\a+1}\cap N$  for every $\a\in N\cap \r$
\end{itemize}
\noindent \emph{models with marker}.\footnote{In the definition of $\mtcl P_\b$, we will assume $\Phi_{\a+1}$ has been defined for all $\a<\b$. While definining $\mtcl P_\b$, we will refer to the notion of model with marker. In that case, the marker $\r$ will be at most $\b$, and hence $\mtcl T_{\a+1}$ will be defined for all $\a\in N\cap\r$.}

If $(N, \r)$ is a model with marker, we will sometimes say that \emph{$\r$ is the marker of $(N, \r)$}.

In our forcing construction, we will use models with markers $(N, \r)$ in a crucial way. The marker $\r$ will tell us that $N$ is to be seen as `active' for all stages in $N\cap\r$.

Given an ordinal $\a$ and a collection $\D$ of models with markers, we let $$\D\av_\a=\{(N, \min\{\a, \r\})\,:\,  (N, \r)\in\D\}$$

Given a set $\D$ of models with markers and an ordinal $\a<\k$, we let $$\D^{-1}(\a)=\{N\,:\, (N, \a)\in\D\}$$ and $$\mtcl N^\D_\a=\{N\in\D^{-1}(\a)\,:\,\a\in N\}$$ We also let $\mtcl M^\D_\a=\{N\in\dom(\D)\,:\,\a\in\cl(N)\}$.

We will make crucial use of the following notion of closed symmetric system
%graph
of models with markers.

 \begin{definition}\label{hom0}
Let $\b<\k$ and let $\D$ be a finite collection of models with markers $(N, \r)$ such that $\r\leq\b$.
%, and let $\mtcl G$ be a collection of edges consisting of models with markers $(N, \r)$ such that $N\in\mtcl W$ and $\r\leq\b$.
%symmetric system of layered models, and let $\mtcl G$ be a collection of nonempty sets of layere models with markers such that
%$\dom(\bigcup\mtcl G)\sub \mtcl W$ and $|e|\leq 2$ for every $e\in\mtcl G$.
We say that $\D$ is a
 \emph{closed $\vec\Phi_\b$-symmetric system (of models with markers)}
if and only if the following holds.
% for $\mtcl W=\dom(\D)$.

\begin{enumerate}
%\item $\mtcl W$ is a closed symmetric system.
\item $(N, \bar\r)\in \D$ for every $(N, \r)\in\D$ and every $\bar\r<\r$.
%\item For every $N\in\dom(\D)$ and every limit ordinal $\r\leq\b$, if $(N, \bar\r)\in\D$ for every $\bar\r<\r$, then $(N, \r)\in\D$.

\item The following holds for every $\a\leq\b$.
\begin{enumerate}
\item $\mtcl M^\D_\a$ is a closed $\Phi_0$-$\a$-symmetric system.
\item If $\a<\b$, then $\mtcl N^\D_{\a}$ is a closed $\Phi_{\a}$-$\a$-symmetric system.
\item For all $N_0$, $N_1\in\mtcl N^\D_{\a+1}$ such that $\d_{N_0}=\d_{N_1}$, for all $\r \leq \alpha+1$, and for every $(M, \r)\in\D\cap N_0$, $(\Psi_{N_0, N_1}(M), \Psi_{N_0, N_1}(\r))\in\D$.
\end{enumerate}
\end{enumerate}
\end{definition}

In the above, we may leave out the parameter $\vec\Phi_\b$ in contexts where it is clear or irrelevant.

Given $\b<\k$ and
 closed
$\vec\Phi_\b$-symmetric systems $\D_0$ and $\D_1$ of models with markers,
%$\b<\k$, a $\vec \Phi_\b$-graph $\mtcl G_0$ of models with markers embedded in $\mtcl W_0$, and a $\vec \Phi_\b$-graph $\mtcl G_1$ of models with markers embedded in $\mtcl W_1$,
we write $$\D_0\cong_{\vec\Phi_\beta} \D_1$$ iff $|\dom(\D_0)|=|\dom(\D_1)|=m$, for some $m<\o$, and there are enumerations $(N^0_i\,:\,i<m)$ and $(N^1_i\,:\,i<m)$ of $\dom(\D_0)$ and $\dom(\D_1)$, respectively, together with countable $M_0$, $M_1\elsub H((2^{{<}\k})^+)$ for which there is an isomorphism
$$\Psi:(M_0; \in, \Phi_\a, \D^*_0, N^0_i)_{i<m, \a<\b}\into (M_1; \in, \Phi_\a, \D^*_1, N^1_i)_{i<m, \a<\b},$$ where for each $i$, $\D^*_i=\D_i\cap M_i$, which is the identity on $M_0\cap M_1$.
%We then define $\mtcl G_0+\mtcl G_1=\mtcl G_0\cup\mtcl G_1\cup\bigcup_{\a<\b}\mtcl G(\a)$, where, given $\a<\b$,  $\mtcl G(\a)$ be the set of all edges of the form $\{(N_0,  \r), (N_1, \r)\}$, where
%\begin{enumerate}
%\item $\r\leq\a+1$,
%\item $\d_{N_0}=\d_{N_1}$,
%\item $(N_0,  \a+1)\in\bigcup\mtcl G_0$,
%\item $(N_1,  \a+1)\in\bigcup\mtcl G_1$, and
%\item $\a\in N_0\cap N_1$.
%\end{enumerate}

The proof of the following lemma is essentially the same as the proof of Lemma \ref{amalg1-}.

\begin{lemma}\label{amalg1} Let $\b<\k$, let $\D_0$ and $\D_1$ be
closed
$\vec\Phi_\b$-symmetric systems of models with markers,
%$\b<\k$,
%$\mtcl G_0$ a
%$\vec \Phi_\b$-graph embedded in $\mtcl W_0$, and $\mtcl G_1$ a
% $\vec \Phi_\b$-graph embedded in $\mtcl W_1$.
%Suppose
and suppose $\D_0\cong_{\vec \Phi_\b}\D_1$. Then $\D_0\cup\D_1$ is a
closed
$\vec\Phi_\b$-symmetric system of models with markers. \end{lemma}

If $R$ is a relation, i.e., a set of ordered pairs, $\dom(R)$ denotes the set $\{a\,:\, (a, b)\in R\mbox{ for some }b\}$ and, given $a\in \dom(R)$, $R``\{a\}$ denotes the set $\{b\,:\, (a, b)\in R\}$. Also, $\range(R)=\bigcup\{R``\{a\}\,:\, a\in \dom(R)\}$.

We will call a function $F$ \emph{relevant} if $\dom(F)\in [S^\k_{\o_2}]^{{<}\o}$ and for every $\a\in \dom(F)$, $F(\a)=(d_\a, b_\a)$, where

\begin{itemize}

 %\item $d_\a$ is a finite relation with $\dom(d_\a)\sub \o_1$ and such that $d_\a``\{\d\}\in [H(\k)]^{{<}\o}$ for each $\d\in \dom(d_\a)$, and
 \item $d_\a$ is a finite relation with $\dom(d_\a)\sub \o_1$ and such that $d_\a``\{\d\}\in [\d\times H(\k)]^{{<}\o}$ for each $\d\in \dom(d_\a)$, and
  \item $b_{\alpha}$ is a function such that
 $\dom(b_{\alpha})\subseteq \o_1$ and such that
 $b_{\alpha}(\delta)< \delta$ for every $\delta\in \dom(b_{\alpha})$.
 \end{itemize}

In the above situation, we will often refer to $d_\a$ and $b_\a$ as, respectively, $d^F_\a$ and $b^F_\a$.

%Given a set $\mtcl G$ of edges of models with markers and an ordinal $\a<\k$, we let $$\mtcl G^{-1}(\a)=\{(N, P)\,:\, (N, P, \a)\in\dom(\bigcup\mtcl G)\}$$ and $$\mtcl N^\mtcl G_\a=\{(N, P)\in\mtcl G^{-1}(\a)\,:\,\a\in N\}$$

%Throughout the paper, given an ordered pair $q=(F, \mtcl G)$, where $F$ is a relevant function and $\mtcl G$ is a collection of edges of models with markers,
%we will often denote $F$ and $\mtcl G$ by, respectively, $F_q$ and $\mtcl G_q$. Given $\a\in \dom(F_q)$, we will denote $d^{F_q}_\a$ and $b^{F_q}_\a$ by, respectively, $d^q_\a$ and $b^q_\a$.

%If $q=(F_q, \mtcl G_q)$, where $F_q$ and $\mtcl G_q$ are as above, and $\b<\k$, we let $\mtcl N^q_\beta$ stand for $\mtcl N^{\mtcl G_q}_\b$.
%Also, if $G$ is a set of ordered pairs $q=(F_q, \mtcl G_q)$ as above, then we let $\mtcl N^G_\b=\bigcup\{\mtcl N^q_\b\,:\, q\in G\}$.

Throughout the paper, given an ordered pair $q=(F, \D)$, where $F$ is a relevant function and $\D$ is a collection of models with markers,
we will often denote $F$ and $\D$ by, respectively, $F_q$ and $\D_q$. Given $\a\in \dom(F_q)$, we will denote $d^{F_q}_\a$ and $b^{F_q}_\a$ by, respectively, $d^q_\a$ and $b^q_\a$.

If $q=(F_q, \D_q)$, where $F_q$ and $\D_q$ are as above, and $\b<\k$, we let $\mtcl N^q_\beta$ stand for $\mtcl N^{\D_q}_\b$. We also let $\mtcl M^q_\beta$ stand for $\mtcl M^{\D_q}_\beta$.
Also, if $G$ is a set of ordered pairs $q=(F_q, \D_q)$ as above, then we let $\mtcl N^G_\b=\bigcup\{\mtcl N^q_\b\,:\, q\in G\}$.

Given an ordinal $\b$ and a set $X$, let $(\b)^X$ denote $$\ssup\{\x+1\,:\,\x\in X\cap\b\}$$ Also, if $q=(F_q, \D_q)$ is as above, we write $\a(q)$ to denote $$\ssup\{\x+1\,:\,\x\in (\dom(F_q)\cup(\bigcup\dom(\D_q)))\cap\a\}$$

%Given a set $\D$ of models with markers, we denote by $\r(\D)$ the maximum of $\{\r\,:\, (M, \r)\in\D\mbox{ for some }M\}$.

%Finally, suppose $\D_0$ and $\D_1$ are set of models with markers, $\bar\D\sub \D_1$, $\r$ is an ordinal such that $\r(\D_0)\leq\r<\k$, and $M$, $M'\elsub H(\k)$ are countable models such that the following holds.
%\begin{enumerate}
%\item $\D_0\in M$ and $\bar\D\in M'$.
%\item $M\cap \r(\D_0)=M'\cap \r(\D_0)$.
%\item $\Psi$ is an isomorphism between $(M; \in, \vec\Phi_\r, \D_0)$ into $(M'; \in, \vec\Phi_\r, \bar\D)$.
%\end{enumerate}
%We then say that $\Psi$ is a \emph{$\r$-reflecting function from $\D_0$ into $\D_1$}.

\subsection{Defining $\mtcl P$}

We will now define our sequence $(\mathcal P_\beta\,:\,\beta\leq\kappa)$ of forcing notions and our sequence $(\Phi_\b\,:\,\b<\k)$ of predicates.
Given $\a\leq\kappa$,
 $\dot G_\alpha$ will be the canonical $\mathcal P_\alpha$-name for the generic filter added by $\mtcl P_\a$.
We shall also denote the forcing relation for $\mtcl P_\a$ by $\Vdash_\a$, and the extension relation for $\mtcl P_\a$ by $\leq_\a$.

Given any $\a\in S^\k_{\o_2}$, and assuming $\mtcl P_\a$ has been defined, we let $\dot C^{\a}$ be some canonically chosen (using $\Phi$) $\mathcal P_{\a}$-name for a club-sequence on $\omega_1^V$ such that $\mathcal P_{\a}$ forces that

\begin{itemize}

\item $\dot C^{\a} = \Phi(\a)$ in case $\Phi(\a)$ is a $\mathcal P_{\a}$-name for a club-sequence on $\omega_1$, and that

\item $\dot C^{\a}$ is some fixed club-sequence on $\o_1$ in the other case.

\end{itemize}

Given $\d\in\Lim(\o_1)$, let $\dot C^\a_\d$ be a name for $\dot C^\a(\d)$.

Let $\beta \leq \kappa$, and suppose $\mathcal P_\alpha$, $\Phi_\a$ and $\Phi_{\a+1}$
have been defined for each $\a<\b$. We aim to define $\mtcl P_\b$ and $\Phi_{\b+1}$, and also $\Phi_\b$ if $\b<\k$ is a nonzero limit ordinal.

An ordered pair $q=(F_q, \D_q)$ is a $\mathcal P_\beta$-condition if and only if it has the following pro\-perties.

\begin{enumerate}

\item $\D_q$ is a closed $\vec\Phi_\b$-symmetric system of models with markers.
\item If $\cf(\b)\geq\o_1$, then $\r<\b$ for every $(N, \r)\in\D_q$.
\item $F_q$ is a relevant function such that $\dom(F_q)\sub \beta\cap\bigcup\mtcl N^q_0$.
%collection of edges of models with markers  with the following properties.
%\begin{enumerate}
%\item For every $(N, \r)\in \D_q$,
%\begin{itemize}
%\item $\r\leq\b$ and
%\item $\r<\b$ if $\cf(\b)\geq\o_1$.
%\end{itemize}
%\item $\mtcl N^q_0$ is a finite
%closed
%$\Phi$-symmetric system of layered models.
%\item $\mtcl G_q$ is a $\vec\Phi_\b$-graph of models with markers embedded in $\mtcl N^q_0$.
%\end{enumerate}
\item For every $\alpha<\beta$ and every $\bar\alpha<\alpha$, \emph{the restriction of $q$ to $\bar\alpha$}, $q\arrowvert_{\bar\alpha}$, is a condition in $\mathcal P_\alpha$, where
 $$q\av_{\bar\a}:=(F_q\restr \bar\alpha, \D_q\av_{\bar\a})$$ Furthermore, if $\cf(\a)\leq\o$, then $q\av_\a\in\mtcl P_\a$.
\item
If $\a\in \dom(F_q)$, then $F_q(\a)=(d^q_\a, b^q_\a)$ satisfies the following.

\begin{enumerate}

%  \item  $\dom(d^q_\a)\cup\dom(b^q_\a)\sub \{\d_N\,:\, N\in\mtcl N^q_{\a+1}\}$ and $\range(d^q_\a)\sub\bigcup\mtcl N^q_0$.

 \item $\dom(b^q_\a)\cup\dom(d^q_\a)\sub \{\d_N\,:\, N\in\mtcl N^q_{\a+1}\}$ and $d^q_\a``\{\d\}\sub\bigcup\{N\in \mtcl N^q_{\a+1}\,:\,\d_N=\d\}$ for every $\d\in \dom(d^q_\a)$.

 % \item  $\dom(b^q_\a)\sub \{\d_N\,:\, N\in\mtcl N^q_{\a+1}\}$ and $\range(d^q_\a)\sub\bigcup\mtcl N^q_0$.

%   \item For all $N_0$, $N_1\in\mtcl N^q_{\a+1}$, if $\d_{N_0}<\d_{N_1}$, $\d_{N_0}\in \dom(d^q_\a)$, and $a\in d^q_\a``\{\d_{N_0}\}$, then there is some $N_1'\in\mtcl N^q_0$ such that $\d_{N_1'}=\d_{N_1}$, $a\in N_1'$, and $\Psi_{N_1', N_1}(a)\in d^q_\a``\{\d_{N_0}\}$.

%  \item For all $N\in\mtcl N^q_{\a+1}$ and all $(\d, a)\in d^q_\a\cap N$, if $M\in\mtcl N^q_{\a+1}$ is such that $\d<\d_{M}<\d_{N}$, then there are $M'$, $N'\in\mtcl N^q_{\a+1}$ such that $\d_{N'}=\d_N$, $\d_{M'}=\d_M$, and $\Psi_{N, N'}(a)\in M'$.  

%\item For every $\d\in \dom(d^q_\a)$ and every $N^*\in\mtcl N^q_{\a+1}$ such that $\d_{N^*}>\d$,
%\begin{enumerate}
%\item $d^q_\a``\{\d\}\sub\bigcup\{N\in\mtcl N^q_{\a+1}\,:\, \d_N=\d_{N^*}\}$, and
%\item for all $N$, $N'\in\mtcl N^q_{\a+1}$ such that $\d_N=\d_{N'}=\d_{N^*}$, $\Psi_{N, N'}`` (d^q_\a``\{\d\})\sub d^q_\a``\{\d\}$.  
%\end{enumerate}

\item For every $\d\in \dom(d^q_\a)$, for all  $N_0$, $N_1\in\mtcl N^q_{\a+1}$ such that $\d_{N_0}=\d_{N_1}=\d$, and for every $(\epsilon, a)\in N_0\cap d^q_\a``\{\d\}$,
\begin{enumerate}
\item $(\epsilon, \Psi_{N_0, N_1}(a))\in d^q_\a``\{\d\}$, and
\item for every $N\in\mtcl N^q_{\a+1}$ such that $\epsilon<\d_N<\d$ there is some $N'\in\mtcl N^q_{\a+1}$ such that $\d_{N'}=\d_N$ and $a\in N'$.   
\end{enumerate}

  \item
 The following holds for every $N\in\mtcl N^q_{\a+1}$ such that $\delta_{N}\in \dom(b^q_{\alpha})$.
 \begin{enumerate}
 \item For every $M\in\mathcal N^q_{\alpha+1}$ such that $b^q_{\alpha}(\delta_{N})< \d_M<\delta_{N}$, $q\av_{\a(q)}\Vdash_\a\d_M\notin\dot C^\a_{\d_{N}}$.
\item $q\av_{\a(q)}$ forces in $\mtcl P_\a$ that for each $a\in N$ there is some $M\in\mtcl T_{\a+1}$ such that $(M, (\a)^M)\in\D_p$ for some $p\in\dot G_\a$, $a\in M$, and $\d_M\in\d_N\setminus\dot C^\a_{\d_N}$.
%  $p\in \dot G_\a$ extending $q\av_{\a(q)}$, some $\a+1$-reflecting function $\Psi$ from $\D_q\av_{\a(q)}$ into $\D_p$, and some $M\in\Psi(N)\cap\mtcl T_{\a+1}$ such that $(M, (\a)^M)\in\D_p$, $\Psi(a)\in M$, and $\d_M\notin\dot C^\a_{\d_N}$.
\end{enumerate}
\end{enumerate}
\end{enumerate}

Given $\mathcal P_\beta$-conditions $q_i$, for $i =0$, $1$, $q_1\leq_\b q_0$ if and only if the following holds.

\begin{enumerate}
\item $\D_{q_0}\sub \D_{q_1}$
\item $\dom(F_{q_0})\sub \dom(F_{q_1})$ and the following holds for every $\alpha\in \dom(F_{q_0})$.

\begin{enumerate}

\item $d^{q_0}_\a\sub d^{q_1}_\a$
\item $b^{q_0}_\alpha\subseteq b^{q_1}_\alpha$
\end{enumerate}

\end{enumerate}

%\begin{enumerate}
%%\item For every $(N, \r)\in \D_{q_0}$ there is some $(N', \r)\in \D_{q_1}$ such that
%%\begin{enumerate}
%%\item $N'\cong N$,
%%\item $\Psi_{N, N'}$ is the identity on $N\cap\r$, and such that
%%\item for every $\a\in N\cap\r$, $\Psi_{N, N'}$  is an isomorphism between $(N; \in, \Phi_{\a+1})$ and $(N_i; \in, \Phi_{\a+1})$.
%%\end{enumerate}
%\item $\dom(F_{q_0})\sub \dom(F_{q_1})$ and for every $\alpha\in \dom(F_{q_0})$, $b^{q_0}_\alpha\subseteq b^{q_1}_\alpha$.
%\item There is a $\r(\D_{q_0})$-reflecting function $\Psi$ from $\D_{q_0}$ into $\D_{q_1}$ such that $\Psi^{q_0}_{q_1}``d^{q_0}_\a\sub d^{q_1}_\a$ for every $\a\in \dom(F_{q_0})$.
%%some $\bar\D\sub \D_{q_1}$ for which there are countable $M$, $M'\in H(\k)$ such that the following holds.
%%\begin{enumerate}
%%\item $\D_{q_0}\in M$ and $\bar\D\in M'$.
%%\item $M\cap \r(\D_{q_0})=M'\cap \r(\bar\D)$.
%%\item Letting $\r=\r(\D_{q_0})$, there is an isomorphism $$\Psi^{q_0}_{q_1}:(M; \in, \vec\Phi_\r, \D_{q_0})\into (M'; \in, \vec\Phi_\r, \bar\D)$$
%%Furthermore, for every $\a\in \dom(F_{q_0})$, $\Psi^{q_0}_{q_1}``d^{q_0}_\a\sub d^{q_1}_\a$.
%\end{enumerate}
%\end{enumerate}

We still need to define $\Phi_{\b+1}$, and $\Phi_\b$ if $\b<\k$ is a nonzero limit ordinal.

Let $\Vdash_\b^*$ denote the restriction of the forcing relation $\Vdash_\b$ for $\mtcl P_\b$ to formulas involving only names in $H(\k)$. Then we let $\Phi_{\b+1}\sub H(\k)$ canonically code the satisfaction relation for the structure $$(H(\k); \Phi_\b, \mtcl P_\b, \Vdash_\b^*)$$

Finally, if $\b<\k$ is a nonzero limit ordinal, we let $\Phi_\b$ be a subset of $H(\k)$ canonically coding $(\Phi_\a\,:\,\a<\b)$.

We may, and will, assume that the definition of $(\Phi_\b\,:\,\b<\k)$ is uniform in $\b$.

\section{Proving Theorem \ref{mainthm}}\label{relevant_facts}
In this section we will prove the lemmas that, together, will yield a proof of Theorem \ref{mainthm}.
 Our first three lemmas are obvious.

\begin{lemma} For all $\a<\b\leq\kappa$, $\mtcl P_\a\sub\mtcl P_\b$.
 \end{lemma}

\begin{lemma}\label{unions} For every $\b\leq\k$, if $\cf(\b)\geq\o_1$, then $\mtcl P_\b=\bigcup_{\a<\b}\mtcl P_\a$. \end{lemma}

\begin{lemma}\label{definability} For every $\b<\k$, $\mtcl P_\b$ and $\Vdash_\b^*$ are definable over the structure $(H(\k); \in, \Phi_{\b+1})$ without parameters. Furthermore, we may assume this definition to be uniform in $\b$. \end{lemma}

%Given $\a<\k$ and $q\in\mtcl P_\k$, let $\a(q)$ denote the supremum of the set of ordinals $\x+1$ such that
%\begin{itemize}
%\item $\x\in \dom(F_q)\cap\a$, and
%\item $\x\in \bigcup(\dom(D_q))\cap\a$. \end{itemize}

%The following lemma is also easy.

%\begin{lemma}\label{compl} Let $\a<\b\leq\k$, let $q\in\mtcl P_\b$, and let $r\in\mtcl P_\a$ be such that $r\leq_\a q\av_{\a(q)}$. Then $$(F_r\cup (F_q\restr [\a,\,\b)),  \D_q\cup\D_r)$$ is a condition in $\mtcl P_\b$ extending both $q$ and $r$. \end{lemma}

%\begin{corollary}\label{compl-cor}
%For all $\a<\b\leq\k$, $\mtcl P_\a$ is a complete suborder of $\mtcl P_\b$.
%\end{corollary}

Given relevant functions $F_0$ and $F_1$, let $F_0+F_1$ denote the function $F$ with domain $\dom(F_0)\cup\dom(F_1)$ such that for every $\a\in \dom(F)$,
\begin{itemize}
\item $F(\a)=F_\epsilon(\a)$ if $\epsilon\in 2$ and $\a\in\dom(F_\epsilon)\setminus \dom(F_{1-\epsilon})$, and
\item $F(\a)=(d^{F_0}_\a\cup d^{F_1}_\a, b^{F_0}_\a\cup b^{F_1}_\a)$ if $\a\in\dom(F_0)\cap\dom(F_1)$.
\end{itemize}

Our next lemma shows in particular that $\mtcl P_\b$ has the $\al_2$-chain condition for each $\b\leq\k$.

\begin{lemma}\label{cc}
For every $\b\leq \kappa$, $\mtcl P_\b$ is $\al_2$-Knaster.
\end{lemma}

\begin{proof} We prove by induction on $\b\leq\k$ that if $(q_\n\,:\,\n<\o_2)$ is a sequence of $\mtcl P_\b$-conditions, then there is $I\in [\o_2]^{\al_2}$ such that $q_{\n_0}$ and $q_{\n_1}$ are compatible in $\mtcl P_\b$ for all $\n_0$, $\n_1\in I$. Let us fix any cardinal $\chi$ such that $\k<\chi$ and let $M_\n$ be, for each $\n<\o_2$, a countable elementary submodel of $H(\chi)$ such that $\vec\Phi_\b$, $q_\n\in M_\n$. By $\CH$ there is $X\in [\o_2]^{\al_2}$ and a set $R$ such that $M_{\nu_0} \cap M_{\nu_1}=R$ for all distinct $\nu_0$, $\nu_1 \in X$.
Using again $\CH$, by shrinking $X$ if necessary we may assume as well that, for some $n$, $m<\o$, there are, for all $\nu\in X$, enumerations $((N^{\nu}_i, \r^\n_i)\,:\, i<n)$ and $(\x^\nu_j\,:\,j<m)$ of $\dom(\D_{q_\nu})$ and $\dom(F_{q_\nu})$, respectively, such that
for all $\nu_0\neq \nu_1$ in $X$ there is an isomorphism $\Psi$ between the structures
 $$\langle M_{\nu_0}; \in,  R,  \D_{q_{\nu_0}}\cap M_{\nu_0}, N^{\nu_0}_i, \r^{\n_0}_i, \x^{\nu_0}_j,  d^{q_{\nu_0}}_{\x^{\nu_0}_j}, b^{q_{\nu_0}}_{\x^{\nu_0}_j}, \vec\Phi_\b\cap M_{\nu_0}\rangle_{i<n, j<m}$$ and  $$\langle M_{\nu_1}; \in,  R,  \D_{q_{\nu_1}}\cap M_{\nu_1}, N^{\nu_1}_i,  \r^{\n_1}_i, \x^{\nu_1}_j,  d^{q_{\nu_1}}_{\x^{\nu_1}_j}, b^{q_{\nu_1}}_{\x^{\nu_1}_j}, \vec\Phi_\b\cap M_{\nu_1}\rangle_{i<n, j<m}$$

We may of course assume that $(\a_{\n_0}; \in, \p_{\nu_0}``R)\cong (\a_{\n_1}; \in, \pi_{\nu_1}``R)$, where $\a_{\n_i}\in\o_1$ is the Mostowski collapse of $M_{\n_i}\cap \Ord$ and $\pi_{\nu_i}$ is the corresponding collapsing function. But then we have that $\Psi$ is the identity on $R\cap\Ord$. This yields that $\Psi$ is the identity on $R\cap H(\k)$ since the function $\Phi:S^\kappa_{\o_2} \into H(\kappa)$ is surjective.

For $\b=0$ the desired conclusion of the lemma follows from Lemma \ref{amalg1-}. Let us now consider the case $\cf(\b)=\o$. We may fix some $\bar\b<\b$ such that $\cf(\bar\b)\leq\o$ and for which there is some $J\in [X]^{\al_2}$ such that $\dom(F_{q_\nu})\sub\bar\b$ for all $\nu\in J$.
By the induction hypothesis for $\bar\b$ we may find $I\in [J]^{\al_2}$ such that $q_{\n_0}\av_{\bar\b}$ and $q_{\n_1}\av_{\bar\b}$ are compatible conditions in $\mtcl P_{\bar\b}$ for any two distinct $\n_0$, $\n_1$ in $I$. For all such $\n_0$ and $\n_1$, let us fix a common extension $q_{\n_0, \n_1}\in\mtcl P_{\bar\b}$ of $q_{\n_0}\av_{\bar\b}$ and  $q_{\n_1}\av_{\bar\b}$. Now, using the fact that $\Psi$ is an isomorphism between the structures
$$\langle M_{\nu_0}; \in,  R,  \D_{q_{\nu_0}}\cap M_{\nu_0}, N^{\nu_0}_i,  \r^{\n_0}_i, \x^{\nu_0}_j,  d^{q_{\nu_0}}_{\x^{\nu_0}_j}, b^{q_{\nu_0}}_{\x^{\nu_0}_j}, \vec\Phi_\b\cap M_{\nu_0}\rangle_{i<n, j<m}$$ and  $$\langle M_{\nu_1}; \in,  R, \D_{q_{\nu_1}}\cap M_{\nu_1}, N^{\nu_1}_i,  \r^{\n_1}_i, \x^{\nu_1}_j,  d^{q_{\nu_1}}_{\x^{\nu_1}_j}, b^{q_{\nu_1}}_{\x^{\nu_1}_j}, \vec\Phi_\b\cap M_{\nu_1}\rangle_{i<n, j<m}$$ fixing $M_{\n_0}\cap M_{\n_1}$ together with Lemma \ref{amalg1}, it follows from the above assumptions that for any two distinct $\n_0$, $\n_1$ in $I$, $$(F_{q_{\n_0, \n_1}}, \D_{q_{\n_0, \n_1}}\cup \D_{q_{\n_0}}\cup \D_{q_{\n_1}})$$ is a condition in $\mtcl P_\b$ extending both $q_{\n_0}$ and $q_{\n_1}$.

The conclusion in the case $\b=\a+1$ follows from the induction hypothesis for $\a$, together with Lemma \ref{amalg1}, by essentially the same argument as in the previous case (with $\a$ for $\bar\b$). In this case, if $I$ is the corresponding set in $[\o_2]^{\al_2}$, we may assume that $\a\in F_{q_\n}$ for all $\n\in I$ (in the other case we may assume that $\a\notin F_{q_\n}$ for all $\n\in I$ and we can then finish by a simpler version of the proof in the present case). If $\n_0<\n_1$ are both in $I$, and $q_{\n_0, \n_1}$ is a common extension of $q_{\n_0}\av_{\a}$ and $q_{\n_1}\av_{\a}$ in $\mtcl P_\a$, $$(F_{q_{\n_0, \n_1}}\cup\{(\a, F_{q_{\n_0}}(\a)+F_{q_{\n_1}}(\a))\}, \D_{q_{\n_0, \n_1}}\cup\D_{q_{\n_0}}\cup\D_{q_{\n_1}})$$ is a common extension of $q_{\n_0}$ and $q_{\n_1}$.

If $\cf(\b)>\o_2$, then by Lemma \ref{unions} we may fix $\bar\b<\b$ such that $q_\n\in \mtcl P_{\bar\b}$ for all $\n\in X$. Hence, the desired conclusion follows in this case from the induction hypothesis for $\bar\b$.

If $\cf(\b)=\o_1$, then there is, again by Lemma \ref{unions}, some $\bar\b<\b$ and some $J\in [X]^{\al_2}$ such that $q_\n\in \mtcl P_{\bar\b}$ for all $\n\in J$, and again we can finish the proof in this case using the induction hypothesis for $\bar\b$.

Finally, suppose $\cf(\b)=\o_2$ and let $\bar\b<\b$ be such that $\cf(\bar\b)\leq\o$ and $R\cap\k\sub\bar\b$.
By the induction hypothesis for $\bar\b$ we may assume, after refining $X$ if necessary, that $q_{\n_0}\av_{\bar\b}$ and $q_{\n_1}\av_{\bar\b}$ are compatible for all distinct $\n_0$, $\n_1$ in $X$. As in the countable cofinality case, let us fix, for all distinct $\n_0$, $\n_1$ in $X$, a condition $q_{\n_0, \n_1}\in\mtcl P_{\bar\b}$ extending both $q_{\n_0}\av_{\bar\b}$ and $q_{\n_1}\av_{\bar\b}$. We may assume that there is no $\b^*\in [\bar\b,\,\b)$ for which there is any $J\in [X]^{\al_2}$ such that $q_\n\in \mtcl P_{\b^*}$ for all $i\in J$ as otherwise we are done by an application of the induction hypothesis for $\b^*$. It follows that there is a strictly increasing sequence $(\b_\x)_{\x<\o_2}$ of ordinals in $\b$ above $\bar\b$ with range cofinal in $\b$ and such that for each $\x<\o_2$ there is exactly one $\n_\x\in X$ such that $(M_{\n_\x}\cap\k)\setminus  \bar\b\sub (\b_\x, \b_{\x+1})$.
But now it is easy to see that for any $\x_0<\x_1<\o_2$, $$(F_{q_{\n_{\x_0}, \n_{\x_1}}}\cup (F_{q_{\n_{\x_0}}}\restriction [\bar\b,\,\b))\cup (F_{q_{\n_{\x_1}}}\restriction [\bar\b, \b)),  \D_{q_{\n_{\x_0}, \n_{\x_1}}}\cup\D_{q_{\n_{\x_0}}}\cup\D_{q_{\n_{\x_1}}})$$ is a condition in $\mtcl P_\b$ extending both of $q_{\n_{\x_0}}$ and $q_{\n_{\x_1}}$.
\end{proof}

Recall that a set $C\sub\k$ is said to be an \emph{$\o_2$-club of $\k$} in case $C$ is unbounded in $\k$ and $\a\in C$ for every $\a\in S^\k_{\o_2}$ such that $\a=\ssup(C\cap\a)$.

Let us fix some large enough cardinal $\chi$ and let $D$ be the set of ordinals of the form $Q\cap\k$, where $Q\prec H(\chi)$ is such that $\Phi\in Q$, $^{\o_1} Q\sub Q$, and $Q\cap\k\in \k$.\footnote{Here is where we use the regularity of $\k$.} For every $\a<\k$ there is some such $Q$ such that $\a\in Q\cap\k$.
Hence, we have the following.

\begin{fact}
$D$ contains an $\o_2$-club of $\k$.
\end{fact}

Let $\vec\Phi=(\Phi_\a\,:\,\a<\k)$.

Our sequence $(\mtcl P_\a\,:\,\a\leq\k)$ is not obviously a forcing iteration, in the sense of $\mtcl P_\a$ being a complete suborder of $\mtcl P_\b$ for all $\a<\b$. However we have the following two lemmas,
which will be enough for our purposes.

\begin{lemma}\label{compl0} For every $\a<\k$, every $q\in\mtcl P_{\a+1}$, and every $r\in\mtcl P_\a$, if $r$ extends $q\av_{\a(q)}$, then $(F_r\cup\{(\a, F_q(\a))\}, \D_q\cup\D_r)\in\mtcl P_\b$.\end{lemma}

\begin{lemma}\label{compl} For every $\a\in D$, $\mtcl P_\a$ is a complete suborder of $\mtcl P_\k$.
\end{lemma}

\begin{proof} It suffices to prove that every maximal antichain $A$ of $\mtcl P_\a$ is predense in $\mtcl P_\k$.

Let $Q\prec H(\chi)$ be a submodel witnessing that $\a\in D$. For every $r\in A$, let $\pi(r)\in \mtcl P_\a\cap Q$ be such that there are countable $M$, $M'\in H(\k^+)$ such that
\begin{itemize}
\item $r\in M$ and $\p(r)\in M'$,
\item $M\cap M'=M\cap Q$, and
\item there is an isomorphism $\Psi_r:(M; \in, r)\into (M'; \in, \pi(r))$ which is the identity on $M\cap M'$.
\end{itemize}

For every $r\in A$ we may find $\p(r)$ as above thanks to the fact that $^{\o}Q\sub Q$.  By Lemma \ref{cc}, $|A|\leq\al_1$. Hence, since in fact $^{\o_1}Q\sub Q$, $\p(A):=\{\p(r)\,:\,r\in A\}\in Q$.

\begin{claim}
$Q\models\p(A)\mbox{ is predense in }\mtcl P_\k$.\end{claim}

\begin{proof}
Otherwise there would be some $q\in \mtcl P_\k\cap Q$ which is incompatible with $\p(r)$ for every $r\in A$. Since $q\in\mtcl P_\a$, there is some $r\in A$ together with a common extension $q'\in\mtcl P_\a$ of $q$ and $r$. Let now 
%$$\D=\bigcup_{\bar\r\leq\r}\{(\Psi_r(N), \bar\r)\,:\, (N, \r)\in \D_{q'}\cap P,\,(P, \r)\in\D_{r}\}$$ 
$$\D_0=\bigcup_{\bar\r\leq\r}\{(\Psi^{-1}_r(N), \bar\r)\,:\, (N, \r)\in \D_{q'}\cap P,\,(P, \r)\in\D_{\p(r)}\}$$ and $$\D_1=\bigcup_{\bar\r\leq\r}\{(\Psi_r(N), \bar\r)\,:\, (N, \r)\in \D_{q'}\cap P,\,(P, \r)\in\D_{r}\},$$
and let 
%$q''=(F_{q'}, \D_{q'}\cup\D_{\p(r)}\cup\D_r\cup\D)$.
$q''=(F_{q'}, \D_{q'}\cup\D_{\p(r)}\cup\D_0\cup\D_1)$. 
Then $q''$ is a condition in $\mtcl P_\k$ extending $q$ and $\p(r)$.\footnote{Getting that  $q''$ is indeed a condition is the reason for having the weak form of the shoulder axiom in clause (4) from Definition \ref{hom-1}, rather than the original full shoulder axiom. A similar remark applies for the last step of the proof of the lemma.}
%why, in the definition of the extension $q_1\leq_\a q_0$, we require not that $\D_{q_0}\sub\D_{q_1}$ but rather that for every $(N, \r)\in\D_{q_0}$ there is some $(N', \r)\in \D_{q_1}$ such that there is a strong enough isomorphism between $N$ and $N'$ which is the identity on $N\cap\r$. We use again this feature of the extension relation in the final step of the proof of this lemma.}
 But that is of course a contradiction.
\end{proof}

It follows, by correctness of $Q$ within $H(\chi)$, that $\p(A)$ is predense in $\mtcl P_\k$. But then $A$ is predense in $\mtcl  P_\k$ as well. To see this, let $q\in \mtcl P_\k$ and let $r\in A$ be such that there is a common extension $q'\in\mtcl P_\k$ of $q$ and $\p(r)$.  Finally, let 
%$$\D=\bigcup_{\bar\r\leq\r}\{(\Psi^{-1}_r(N), \bar\r)\,:\, (N, \r)\in \D_{q'}\cap P,\,(P, \r)\in\D_{\p(r)}\}$$ and let %$q''=(F_{q'}, \D_q\cup\D_{\p(r)}\cup\D_r\cup\D)$. 
 $$\D_0=\bigcup_{\bar\r\leq\r}\{(\Psi^{-1}_r(N), \bar\r)\,:\, (N, \r)\in \D_{q'}\cap P,\,(P, \r)\in\D_{\p(r)}\}$$ and $$\D_1=\bigcup_{\bar\r\leq\r}\{(\Psi_r(N), \bar\r)\,:\, (N, \r)\in \D_{q'}\cap P,\,(P, \r)\in\D_r\},$$  and let 
$q''=(F_{q'}, \D_q\cup\D_r\cup\D_0\cup\D_1)$.
Then $q''$ is a condition in $\mtcl P_\k$ extending $q$ and  $r$.\end{proof}

\subsection{Properness}\label{subsection-properness}
Given a countable elementary substructure $N$ of $H(\kappa)$ and a $\mtcl P_\beta$-condition $q$, for some $\b\leq\kappa$, we will say that \emph{$q$ is $(N, \mtcl P_\beta)$-generic} if and only if $q$ forces $\dot G_\beta\cap A\cap N\neq\emptyset$ for every maximal antichain $A$ of $\mtcl P_\beta$ such that $A\in N$. Note that this is more general than the standard notion of $(N, \mtbb P)$-genericity, for a forcing notion $\mtbb P$, which applies only if $\mtbb P\in N$. Indeed, in our situation $\mtcl P_\beta$ is of course never a member of $N$ if $N\sub H(\kappa)$.

%Given an ordinal $\b$ and a set $X$, let $(\b)^X$ denote $$\ssup\{\x+1\,:\,\x\in X\cap\b\}$$

The properness of $\mtcl P_\b$, for all $\b\leq\kappa$, will be an immediate consequence of Lemmas \ref{properness1} and \ref{properness2} (and \ref{cc}).

\begin{lemma}\label{properness1} Let $\b<\kappa$, $q\in \mtcl P_{\b}$, and let $N$ be a countable elementary submodel of $H(\k)$ such that $q\in N$, $\b\in N$, and $N\in\mtcl T_{\x+1}$ for each $\x\in N\cap\b$.
% Let $P\sub (\Ord\times N)\times N$ be such that $P(\x+1)=\Phi_{\x+1}\cap N$ for every $\x\in N\cap (\b)^N$.
Then there is an extension $q^*\in\mtcl P_\b$ of $q$ such that $(N, (\b)^N)\in\D_{q^*}$.
\end{lemma}

\begin{proof}
 It is straightforward to see that $$q^*=(F_q,  \D_q\cup\{(N, \r)\,:\,\r\leq (\b)^N\})$$ is such an extension.
\end{proof}

Our main properness lemma is the following.

\begin{lemma}\label{properness2} Let $\beta<\kappa$, $q\in\mtcl P_\b$, $N\in\mtcl T_{\b+1}$,
%$P^*\sub N$,
and suppose $(N, (\b)^N)\in\D_q$. Then $q$ is $(N, \mtcl P_\b)$-generic.
\end{lemma}

\begin{proof}
The proof is by induction on $\b$.
Let $A\in N$ be a maximal antichain of $\mtcl P_\b$, and suppose that $q$ extends some condition $r_0\in A$. We want to prove that $r_0\in N$, and for this it suffices to show that $q$ is compatible with a condition in $A\cap N$.
The case $\b=0$ follows at once from Lemma \ref{amalg0-} and the definition of $\mtcl T_1$, so we will assume in what follows that $\b>0$.

Let us next consider the case when $\b$ is a successor ordinal $\a+1$.
We may assume that $\alpha \in \dom(F_q)\cap N$ and $\d_N\in \dom(b^{q}_{\a})$, as otherwise the proof is a simpler version of the proof in this case.
%By extending $q$ if necessary, we may assume that for every $\d<\d_N$ and every $a$ such that $(\d, a)\in d^q_\a$ there is some $N'\in\mtcl N^q_{\a+1}$ such that $\d_{N'}\geq\d_N$ and $a\in N$.
Let $\dot B$ be a $\mtcl P_\a$-name for the (partially defined) function on $A\times\o_1$ sending a pair $(r, \eta)$ to some condition $q^r_\eta\in\mtcl P_\b$ with the following properties (provided there is some such $q^r_\eta$; otherwise $\dot B$ is not defined at $(r, \eta)$).

\begin{enumerate}
\item $q^r_\eta$ extends $r$.
\item $q^r_\eta\av_{\a(q^r_\eta)}\in \dot G_\a$
\item $\a\in \dom(F_{q^r_\eta})$, $d^q_\a\cap N\sub d^{q^r_\eta}_\a$,
and $b^q_\a\cap N\sub b^{q^r_\eta}_\a$.
\item $\D_q\cap N\sub\D_{q^r_\eta}$
\item For every $Q\in \mtcl N^q_\b\cap N$, $Q\cap\mtcl N^q_\b=Q\cap\mtcl N^{q^r_\eta}_\b$.
%\item $\{(Q, P, \r)\in\bigcup\mtcl G_{q^r_\eta}\cap N\,:\, Q\in Q'\}=\{(Q,P, \r)\in \bigcup\mtcl G_q\cap N\,:\, Q\in Q'\}$ for every $(Q', P', \r')\in \bigcup\mtcl G_q\cap N$.
\item $\d_Q>\eta$ for every $Q\in\mtcl N^{q^r_\eta}_\b$ such that $\d_Q\neq \d_R$ for any $R\in\mtcl N^q_\b\cap N$.
\end{enumerate}

By the $\al_2$-c.c.\ of $\mtcl P_\b$ and $\k^{\al_1}=\k$ we may in fact take $\dot B$ to be in $H(\k)$.  Also, by Lemma \ref{definability} and since $N\in\mtcl T_{\b+1}$ and $A\in N$, we may assume that $\dot B\in N$.

By an instance of clause (5)(c)(ii) in the definition of condition, we may find an extension $q'\in\mtcl P_\a$ of $q\av_{\a(q)}$ for which there is some $M\in\mtcl N^{q'}_\a\cap\mtcl T_{\a+1}$ such that  $(M, (\a)^M)\in\D_{q'}$, $\d_M<\d_N$, $\dot B \in M$, and $q'\Vdash_\a \d_M\notin\dot C^\a_{\d_N}$.
%$M\in\mtcl T_\b$ such that $(M, (\a)^M)\in\D_{q'}$, $\Psi(\dot B)\in M$, where $\Phi$ is a reflecting function from $\D_{q_0}$ into $\D_{q_1}$, $\d_M<\d_N$, and $q'\Vdash_\a\d_M\notin \dot C^\a_{\d_N}$.
 By extending $q'$ if necessary, and using the openness (in $V^{\mtcl P_\a}$) of $\dot C^\a_\d$ for every $\d\in\Lim(\o_1)$, we may assume that there is some $\eta^*<\d_M$ such that $q'\Vdash_\a [\eta^*,\,\d_M)\cap \dot C^\a_\d=\emptyset$ whenever $\d\in \dom(b^q_\a)$ is such that $b^q_\a(\d)<\d_M<\d$.

Let now $G$ be $\mtcl P_\a$-generic and such that $q'\in G$. Working in $M[G]$, we may find some $r^*\in A$ such that $q^*:=\dot B_G(r^*, \eta^*)$ is defined. Indeed, the existence of such an $r^*$ is witnessed by $r_0$---as, in turn, witnessed by $q$---and, furthermore, the existence of such an $r^*$ can be expressed over the structure $(H(\k)[G]; \in)$ by a sentence with objects in $M[G]$ as parameters (namely, $G$, $\dot B$, and $\eta^*$).

By the induction hypothesis for $\a$ we have that $M[G]\cap V=M$ and therefore $q^*\in M$. In particular, we also have that $r^*\in M$. But then $r^*\in N$ since $|A|\leq\al_1$ by Lemma \ref{cc} and $\d_M<\d_N$. Since $q^*\av_{\a(q^*)}\in G$, we may find a common extension $q''\in\mtcl P_\a$ of $q^*\av_{\a(q^*)}$ and $q'$ forcing the above facts.
%Let $\mtcl N$ be the union of
%\begin{itemize}
%\item $\mtcl N^q_\b$ and
%\item the set of layered models of the form $$(\range(\Psi^{\mtcl N^q_0}_{Q, N'}), \Psi^{\mtcl N^q_0}_{Q, N'}``P),$$ where $Q\in\mtcl N^{q^*}_\b$, $N'\in\mtcl N^q_\b$, and $\d_Q\leq \d_{N'}$.
%\end{itemize}
We may now amalgamate $q''$, $q^*$ and $q$ into a common extension $q^{**}\in\mtcl P_\b$. In order to do so, it suffices to let $$q^{**}=(F_{q''}\cup\{(\a, (d^q_\a\cup d^\dag,  b^q_\a\cup b^{q^*}_\a))\}, \D_{**}),$$ where $$\D_{**}=\D_{q''}\cup\D_q\cup\{(\Psi_{N, N'}(Q), \b), N'\in\mtcl N^q_\b,\,Q\in\mtcl N^{q^*}_\b,\,\d_{N'} =\d_N\},$$
%$\mtcl G_{**}$ is the union of
%\begin{itemize}
%\item  $\mtcl G_{q''}\cup\mtcl G_q$ and
%\item the set of edges of models with markers of the form $\{(\range(\Psi^{\mtcl N^q_0}_{Q, N'}), \b)\}$, where $Q\in\mtcl N^{q^*}_\b$, $N'\in\mtcl N^q_\b$, and $\d_Q\leq \d_{N'}$,
%\end{itemize}
%\noindent
and where $d^\dag$ is the relation with $\dom(d^\dag)=\dom(d^{q^*}_\a)$ such that $$d^\dag``\{\d\}=\bigcup\{\Psi_{N_0, N_0'}``(d^{q^*}_\a``\{\d\})\,:\, N_0\in\mtcl N^{q^*}_\b,\,N_0'\in\mtcl N^q_\b,\,\d_{N_0} =\d_{N_0'}=\d\}$$
%$$d^\dag``\{\d\}=\bigcup\{\Psi_{N, N'}``(d^{q^*}_\a``\{\d\})\,:\, N'\in\mtcl N^q_\b,\,\d_{N'} =\d_N\}$$ 
for every $\d\in \dom(d^{q^*}_\a)$.

The following claim is immediate.

\begin{claim}\label{cl000}
The following are true.
\begin{enumerate}
\item $\{(Q, \r)\in\D_{**}\,:\,\r\leq\a\}=\D_{q''}$
\item $\D_{**}$ is a closed $\vec\Phi_{\b}$-symmetric system of models with markers.
\end{enumerate}
\end{claim}

%\begin{proof} (1) holds trivially by construction of $\D_{**}$ and since $q''$ extends $q\restriction\a$.

%(2) follows immediately using the construction of $\D_{**}$ and the fact that $\D_{q}$ and $\D_{q^*}$ are closed  $\vec\Phi_{\b}$-symmetric systems.
%\end{proof}

Given the conclusion of the claim above, in order to prove that $q^{**}$ is a condition in $\mtcl P_\b$ it remains only to show that the relevant instances of clauses (5)(b) and (5)(c)(i) in the definition of condition holds for it. (5)(b) follows from the fact that $d^q_\a\cap N\sub d^{q^*}_\a$ and the construction of $d^\dag$.
Finally, (5)(c)(i) follows by the choice of $\eta^*$ and the fact that $\d_Q>\eta^*$ for every $Q\in \mtcl N^{q^*}_\b$ such that $\d_Q\notin\{\d_R\,:\, R\in\mtcl N^q_\b\cap N\}$.

This concludes the proof in this case since $r^*$ is a condition in $A\cap N$ weaker than $q^{**}$.

Let us now prove the lemma in the case when $\b$ is a nonzero limit ordinal.

Suppose first that $\cf(\b)>\o_1$. Since $|A|\leq\al_1$ by Lemma \ref{cc},\footnote{We may of course assume $|A|=\al_1$ since otherwise $r_0\in A\sub N$ and hence we are already done in that case.} we may find $\bar\b<\b$ such that $A$ is a maximal antichain of $\mtcl P_{\bar\b}$. Taking $\bar\b$ to be the least such ordinal, it follows that $\bar\b\in N$. But then we have that $q$ is $(N, \mtcl P_{\bar\b})$-generic by the induction hypothesis for $\bar\b$, and therefore $r_0\in N$.

We now move on to the case $\cf(\b)\leq\o_1$. Let $\bar\b=\ssup(N\cap\b)$, and notice that $\bar\b=\b$ if $\cf(\b)=\o$ and $\bar\b<\b$ if $\cf(\b)=\o_1$.
Let $E=E^0_\b$ if $\b\in S^\k_\o$ and $E=E^1_\b\cap\bar\b$ if $\cf(\b)=\o_1$.

\begin{claim}\label{E} $E\sub Q$ for every $Q\in \mtcl M^q_{\bar\b}$ such that $\d_N\leq\d_Q$. \end{claim}

\begin{proof}
Let us first consider the case $\cf(\b)=\o$. Since $(Q; \in, \vec E)\elsub (H(\k); \in, \vec E)$, it is enough to show that $\b\in Q$. If $\d_Q=\d_N$, then this follows from the fact that $\b\in N$ 
%together with the fact that $\mtcl N^q_\b$ is a closed symmetric system
 and therefore $\Psi^*_{N, Q}$ is  an elementary embedding between $(\cl(N); \in, N)$ and $(\cl(Q); \in, Q)$ which is the identity on $\cl(N)\cap\cl(Q)$. Now suppose $\d_N<\d_Q$. Since $N$, $Q \in \mtcl M^q_\b$ and $\b \in N$, there are $N'$, $Q'\in\mtcl M^q_\b$ such that $\d_{N'}=\d_N$, $\d_{Q'}=\d_Q$, $N'\in Q'$ and $\b \in N'$. But then $\b\in Q'$ and again $\b\in Q$ since $\Psi^*_{Q', Q}$ is an elementary embedding between $(\cl(Q'); \in, Q')$ and $(\cl(Q); \in, Q)$ which is the identity on $\cl(Q')\cap\cl(Q)$.

Let us now consider the case $\cf(\b)=\o_1$. If $\d_N<\d_Q$, then by a similar argument as in the previous case 
%using the fact that $\mtcl N^q_{\bar\b}$ is a closed symmetric system, 
we may find $N' \in\mtcl M^q_{\bar\b}$ such that $N'\in Q$, $\d_{N'}=\d_N$ and $N'\cap\bar\b$ is cofinal in $\bar\b$. Let $N'=Q$ if $\d_Q=\d_N$. It follows, in either case, that $\b^*:=\min(N'\setminus (\bar\b+1))$ exists and $\cf(\b^*)=\o_1$. But then $\bar\b$ is a limit point of both $E^1_\b$ and $E^1_{\b^*}$. Hence, by coherence of $\vec E^1$, $E=E^1_\b\cap\bar\b=E^1_{\b^*}\cap \bar\b$, and therefore $E\sub N'$ since in fact $E_{\b^*}\cap\bar\b=E_{\b^*}\cap N'$. But now we are done since $N'\sub Q$.
\end{proof}

Let now $\s\in E\cap\bar\b$ be
%above all ordinals in $\dom(F_q)\cap\bar\b$.
\begin{enumerate}
\item above $\dom(F_q)\cap\bar\b$ and
\item such that for every $Q\in\mtcl N^q_0$,
% such that $\d_Q\geq\d_N$,
%$N'\in \dom(\D_q)$, if $\d_{N'}\geq\d_N$ and $X_Q$ is bounded in $\bar\b$, with $X_{N'}$ being the set of $\a\in N\cap\bar\b$ such that $(N', \bar\a+1)\in\D_q$ and $\Psi^{\mtcl N^q_0}_{N, N'}$ is an elementary embedding from $(N; \in, \Phi_{\a+1})$ into $(N'; \in, \Phi_{\bar\a+1})$, where $\bar\a=\Psi^{\mtcl N^q_0}_{N, N'}(\a)$, then $\s$ is a bound of $X_{N'}$.
\begin{enumerate}
\item if $Q\cap\bar\b$ is bounded in $\bar\b$, then $\s>\ssup(Q\cap\bar\b)$, and
\item if $R_Q:=\{\r<\bar\b\,:\, Q\in\mtcl N^q_\r\}$ is bounded in $\bar\b$, then $\s>\ssup(R_Q)$.
%\footnote{Note that, according to the second clause of Definition \ref{hom0}, $Q\in\mtcl N^q_{\bar{\beta}}$ whenever $R_Q$ is unbounded in $\bar\b$.}
\end{enumerate}
\end{enumerate}

For convenience we pick $\s$ such that $\cf(\s)\leq\o$.

Let $\mtcl N$ be the set of $N'\in\mtcl M^q_{\bar\b}$ such that $\d_{N'}\geq\d_N$ and $N'\in\mtcl N^q_\a$ for cofinally many $\a<\bar\b$. By Claim \ref{E} we know that $\s\in N'$ for every $N'\in\mtcl N$.

% $Q\in \mtcl N^q_0$ such that $\d_N\leq\d_Q$ and $Q\cap\bar\b$ is cofinal in $\bar\b$.

Let
$G$ be $\mtcl P_\s$-generic and such that $q\av_\s\in G$. Working in $N[G]$, we may find a condition $q^*\in\mtcl P_\b$ with the following properties.
\begin{enumerate}
\item $q^*$ extends some $r^*\in A$.
\item $q^*\av_\s\in G$.
\item $\D_q\cap N\sub\D_{q^*}$.
\item $M\cap \D_q =M\cap \D_{q^*}$ for every $M\in \dom(\D_q) \cap N$.
\end{enumerate}

We may indeed find a $q^*$ with these properties since the existence of such a condition is witnessed by $q$ and can be expressed over the structure $(H(\k)[G]; \in, \Phi_{\b+1})$ by a sentence with parameters in $N[G]$.
%---namely $G$, $A$, and $\mtcl G_q\cap N$.
% finitely many sets of the form $(\mtcl N, \r)$ with $\mtcl N\sub\mtcl N^q_0\cap N$ and $\r\in \b\cap N$.
By the induction hypothesis applied to $\s$ we know that $N[G]\cap V=N$ and therefore $q^*\in N$, and in particular $r^*\in N$. It thus suffices, in order to finish the proof, to show that $q$ and $q^*$ are compatible.
Towards this, we first fix a common extension $q'\in\mtcl P_\s$ of $q\av_\s$ and $q^*\av_\s$ forcing the above facts.

Let us say that a collection $\D$ of models with markers is a
 \emph{closed $\vec\Phi_\b$-symmetric system above $\s$}
if and only if the following holds.
% for $\mtcl W=\dom(\D)$.

\begin{enumerate}
%\item $\mtcl W$ is a closed symmetric system.
\item $(N, \bar\r)\in \D$ for every $(N, \r)\in\D$ and every $\bar\r<\r$.

\item The following holds for every $\a<\b$ such that $\s<\a$.
\begin{enumerate}
\item $\mtcl N^\D_{\a}$ is a closed $\Phi_{\a}$-$\a$-symmetric system.
\item For all $N_0$, $N_1\in\mtcl N^\D_{\a+1}$ such that $\d_{N_0}=\d_{N_1}$ and for every $(M, \r)\in\D\cap N_0$ with $\r\leq\a+1$, $(\Psi_{N_0, N_1}(M), \Psi_{N_0, N_1}(\r))\in\D$.
\end{enumerate}\end{enumerate}

\begin{claim}\label{o} There is $\D^{**}$, a closed $\vec\Phi_\b$-symmetric system above $\s$, of the form $\D_q\cup\D_{q^*}\cup\bigcup_{\bar\r<\b}\D\av_{\bar\r}$, where $\D$ is a set of models with markers $(M, \r)$ such that $\r\leq\bar\b$. 
%$(\Psi_{N, N'}(M), \Psi_{N, N'}(\r))$ for some $\a\in (\s,\,\bar\b)\cap N$, $N'\in\mtcl N^q_{\a+1}\cap\mtcl N$, $(M, \r)\in\D_{q^*}$, and $\r\leq\a$.
\end{claim}

\begin{proof}
$\D$ is constructed as $$\bigcup\{\D_\a\,:\,\s<\a\leq\bar\b,\,\a\in N\},$$ where $(\D_\a\,:\,\s<\a\leq\bar\b,\,\a\in N)$ is the sequence of sets of models with markers defined by letting $\D_{\s+1}$ be the union of $\D_q^{-1}(\s+1)$ and $$\{(\Psi_{N, N'}(M), \bar\r)\,:\, (M, \r)\in\D_{q^*}\av_{\s+1},\,N'\in\mtcl N^q_{\s+1}\cap\mtcl N,\,\bar\r\leq\Psi_{N, N'}(\r)\}$$ and, given any $\a\in N\cap (\s+1,\,\bar\b]\cap N$ such that $\D_{\a'}$ has been defined for all $\a'\in (\a\cap N)\setminus (\s+1)$, letting 
%$\D_\a=\bigcup\{\D_{\a'}\,:\,\a'\in \a\cap N\setminus (\s+1)\}$ if $\a$ is a limit ordinal and, if $\a=\bar\a+1$, letting 
$\D_\a$ be a closed $\vec\Phi_\a$-symmetric system which is the union of $$\bigcup\{\D_{\a'}\,:\,\a'\in \a\cap N\setminus (\s+1)\}\cup\D_q^{-1}(\a),$$ the set $$\D^\a:=\{(\Psi_{N, N'}(M), \bar\r)\,:\, (M, \r)\in\D_{q^*}\av_{\a},\,N'\in\mtcl N^q_{\a}\cap\mtcl N,\,\bar\r\leq\Psi_{N, N'}(\r)\},$$ and a set $\D(\a)$ of models of the form $(\Psi_{N_0, N_1}(M), \bar\r)$, for certain $N_0$, $N_1\in\mtcl N$, $(M, \r)\in \D^\a\cap N_0$ and $\bar\r\leq\Psi_{N_0, N_1}(\r)$. Such a set $\D(\a)$ 
%can be found by iterated applications of Lemma \ref{amalg0-}; i.e., $\D(\a)$ 
can be constructed as $\bigcup\{\D(\a, n)\,:\, n<\o\}$, where $\D(\a, 0)=\D^\a$ and, for each nonzero $n$, $\D(\a, n)$ is the result of taking copies of $\D(\a, n-1)\cap N'$, for some $N'\in\mtcl N$ and some $\bar\a<\a$ such that $N'\in\mtcl N^q_{\bar\a+1}$, into other models in $\mtcl N\cap \mtcl N^q_{\bar\a+1}$. 
\end{proof}

Let $\D^{**}$ be as in Claim \ref{o}. By the construction of $\D^{**}$, the choice of $\s$, and the fact that $q'$ extends both $q\av_\s$ and $q^*\av_\s$, it is clear that $\D_{q'}\supseteq\D^*\av_\s$. But then $\D_{q'}\cup\D^{**}$ is a closed $\vec\Phi_\b$-symmetric system.

Let us finally define $$q^{**}=(F_{q'}\cup f^\dag\cup (F_q\restriction [\bar\b,\,\b)),  \D_{q'}\cup\D^{**}),$$ where $f^\dag$ is the function with domain $\dom(F_{q^*})\setminus\s$ obtained by letting $f^\dag(\a)=(d^\dag_\a, b^{q^*}_\a)$, for $d^\dag_\a$ being the relation with domain $\dom(d^{q^*}_\a)$ defined by letting $d^\dag_\a``\{\d\}$ be
%\noindent
$$d^\dag``\{\d\}=\bigcup\{\Psi_{N_0, N_0'}``(d^{q^*}_\a``\{\d\})\,:\, N_0\in\mtcl N^{q^*}_\b,\,N_0'\in\mtcl N^q_\b,\,\d_{N_0} =\d_{N_0'}=\d\}$$
%$$\bigcup\{\Psi_{N_0, N_0'}``(d^{q^*}_\a``\{\d\})\,:\, N_0\in\mtcl N^q_{\a+1}\cup\mtcl N^{q^*}_{\a+1},\,N_0'\in\mtcl N^q_{\a+1},\,\d_{N_0} =\d_{N_0'}\}$$ 
We then have by the above that $q^{**}$ is a condition in $\mtcl P_\b$. This concludes the proof since of course $q^{**}$ extends $q$ and $q^*$.
%and where $$\mtcl G_{**}=\mtcl G_{q'}\cup\bigcup_{n<\o}\overline{\mtcl A}_n(\mtcl G_q, \mtcl G_{q^*}, \mtcl N^{q'}_0)$$
%Such a void expansion of $\D_{q'}\cup\overline{\mtcl A}(\D_q, \D_{q^*})$ exists by Fact \ref{void} since $\D_{q'}\cup\overline{\mtcl A}(\D_q, \D_{q^*})$ is a weakly $\vec\Phi_\b$-symmetric system embedded in $\mtcl N^{q'}_0$, which follows from Lemma \ref{amalg0} together with Claim \ref{pre-amalg-claim}.
%By Lemma \ref{amalg0}, $\bigcup_{n<\o}\overline{\mtcl A}_n(\mtcl G_q, \mtcl G_{q^*}, \mtcl N^{q'}_0)$ is a layered $\vec\Phi_\b$-graph embedded in $\mtcl N^{q'}_0$. Hence so is $\mtcl G_{q'}\cup\bigcup_{n<\o}\overline{\mtcl A}_n(\mtcl G_q, \mtcl G_{q^*}, \mtcl N^{q'}_0)$.
%It follows that $q^{**}$ is indeed a condition. This concludes the proof since of course $q^{**}$ extends $q$, $q^*$, and $q'$.
\end{proof}

Corollary \ref{properness} follows from Lemmas \ref{properness1}, \ref{properness2}, and also Lemma \ref{cc} for $\b=\k$.

\begin{corollary}\label{properness} For every $\b\leq\kappa$, $\mtcl P_\b$ is proper.
\end{corollary}

Given a $\mtcl P_\k$-generic filter $G$ and $\a<\k$, let $$D^G_\a=\{\d_N\,:\, N\in\mtcl N^G_{\a+1}\}$$ Let also $\dot D_\a$ be a $\mtcl P_{\a+1}$-name for $D^G_\a$.

\begin{lemma}\label{new-reals} $\mtcl P_\kappa$ adds $\kappa$-many reals.
\end{lemma}

\begin{proof}
There are many ways to prove this lemma. The following argument is one of them.

Let $(\a_i)_{i<\k}$ be the strictly increasing enumeration of the set of ordinals $\a\in S^\k_{\o_2}$ such that $\Phi(\a)$ is a $\mtcl P_\a$-name for a ladder system---i.e., each $\dot C^\a_\d$ is forced to have order type $\o$. Let $\dot G$ be the canonical $\mtcl P_\k$-name for the generic object. Now, for every $i<\k$, let $\dot r_i$ be a name for the set of $n<\o$ with the property that if $\d$ is the $n$-th member, in its increasing enumeration, of the domain of $$\bigcup\{b^q_{\a_i}\,:\, q\in \dot G,\,\a_i\in \dom(F_q)\},$$ $q\in\dot G$ is such that $\a_i\in\dom(F_q)$, and $\d\in \dom(b^q_{\a_i})$, then $b^q_{\a_i}(\d)$ is of the form $\bar\d+k$, where $\bar\d\in\Lim(\o_1)$ and $k$ is an even natural number. Standard density arguments show that each $\dot r_i$ is a name for an infinite subset of $\o$---since each $\dot C^{\a_i}$ is forced to be a ladder system and hence we can in fact put $\d$ in the domain of $b_{\a_i}$ whenever $\d$ is in $D^G_\a$---and that $\Vdash_\k\dot r_i\neq \dot r_{i'}$ if $i\neq i'$.
\end{proof}

On the other hand, since $\mathcal P_\k$ has cardinality $\k$ and has the $\al_2$-chain condition, there are $\k^{\al_1}=\kappa$ many nice names (s.\ \cite{KUNEN}) for subsets of $\omega_1$, and hence $\mathcal P_\kappa$ forces $2^{\aleph_0}=2^{\aleph_1}=\kappa$.

\subsection{Measuring}\label{subsection-measuring}

\begin{lemma}\label{measuring} $\mathcal P_\kappa$ forces $\Measuring$.
\end{lemma}

\begin{proof}
Let $G$ be $\mathcal P_\kappa$-generic and let $\vec C=(C_\delta\,:\,\delta\in\Lim(\omega_1))\in V[G]$ be a club-sequence on $\omega_1$. We want to see that there is a club of $\omega_1$ in $V[G]$ measuring $\vec C$.
By the $\al_2$-c.c.\ of $\mtcl P$ together with $\kappa^{\aleph_1}=\kappa$, we may assume that $\vec C=\dot C_G$ for some $\mtcl P$-name $\dot C\in H(\kappa)$ for a club-sequence on $\o_1$. Again by the $\al_2$-c.c.\ of $\mtcl P$, together with Lemma \ref{compl} and the fact that $\Phi^{-1}(\dot C)$ is stationary in $S^\kappa_{\o_2}$, we may fix some $\a\in D$ such that $\dot C$ is in fact a $\mtcl P_\a$-name and $\Phi(\a)=\dot C$.
We then have that $\vec C=\Phi(\a)_G$.

It is immediate to see that $D^G_\a$ is unbounded in $\omega_1$. In fact, given any condition $q$ and any countable sufficiently correct $N\elsub H(\k)$ such that $q$, $\a\in N$, we may find by Lemma \ref{properness1} an extension $q^*\in\mtcl P_\k$ of $q$ such that $(N, \a+1)\in N$, and every such condition forces that $\d_N\in \dot D_\a$.

\begin{claim}\label{closure} $D^G_\a$ is closed in $\o_1$.
\end{claim}

\begin{proof}
%Suppose, towards a contradiction, that $q\in\mtcl P_\kappa$ and $\d\in\Lim(\o_1)$ are such that $q$ forces $\d$ to be a limit point of $\dot D_\a$, but there is no extension $q'\in\mtcl P_{\kappa}$ of $q$ for which there is some $N\in\mtcl N^{q'}_{\a+1}$ such that $\d_N=\d$. We may of course assume, without loss of generality, that there is some model $M\in\mtcl N^q_{\a+1}$ such that $\d_M<\d$. Let $M$ be of maximal height among such models. We may assume that $\d_M\in \dom(d^q_\a)$. Now we may extend $q$ to a condition $q'$ obtained by adding $\d$ to $d^q_\a``\{\d_M\}$. But that yields a contradiction since then $q'$ forces, by clause (5)(b)(i) in the definition of condition, that $\dot D_\a\cap \d$ is bounded by $\d_M+1<\d$.
Suppose, towards a contradiction, that $q\in\mtcl P_\kappa$ and $\d\in\Lim(\o_1)$ are such that $q$ forces $\d$ to be a limit point of $\dot D_\a$, but there is no extension $q'\in\mtcl P_{\kappa}$ of $q$ for which there is some $N\in\mtcl N^{q'}_{\a+1}$ such that $\d_N=\d$. By Lemma \ref{properness1} we may assume, without loss of generality, that there is some model $N\in\mtcl N^q_{\a+1}$ such that $\d_N>\d$. Let $N$ be of minimal height among such models. We may assume that $\d_N\in \dom(d^q_\a)$.  Now we may extend $q$ to a condition $q'$ obtained by adding $(\epsilon, \d)$ to $d^q_\a``\{\d_N\}$, where $\epsilon<\d$ is above $\d_M$ for every $M\in\mtcl N^q_{\a+1}$ such that $\d_M<\d_N$. But that yields a contradiction since then $q'$ forces, by clause (5)(b)(i) in the definition of condition, that $\dot D_\a\cap \d$ is bounded by $\epsilon<\d$.
\end{proof}

Given any $q\in G$ such that $\alpha\in \dom(F_q)$ and any limit ordinal $\d\in D^G_\a$, if $\d\in \dom(b^q_\alpha)$, then $D^G_\alpha\cap (b^q_\a(\d),\, \d)$ is disjoint from $C_\d$. Hence, in order to finish the proof of the lemma it is enough to show that if $q\in G$ is such that $\a\in \dom(F_q)$, $N \in\mtcl N^q_{\a+1}$, and there is no $q'\in G$ extending $q$ and such that $\d_{N}\in \dom(b^{q'}_\a)$, then a tail of $D^G_\a \cap \delta_N$ is contained in $C_{\d_{N}}$. Without loss of generality, we may and will assume that $\delta_N$ is a limit point of $D^G_\a$.

So, let $q$ be a condition with $\a\in \dom(F_q)$ and let $N \in\mtcl N^q_{\a+1}$ be such that $\d_{N}\notin\dom(b^{q'}_\a)$ for any $q'\in\mtcl P$ extending $q$. It suffices to find an extension $q^*$ of $q$ in $\mtcl P$ and some $\epsilon <\d_N$ with the property that if $q'\in\mtcl P$ extends $q^*$ and $M\in\mtcl N^{q'}_{\a+1}$ is such that $\epsilon <\d_M<\d_N$, then $q'\av_{\a(q')}\Vdash_\a \d_M\in \dot C^\a_{\d_N}$.

%By our choice of $q$ and Lemma \ref{compl} together with the fact that $\a\in D$, we may assume, after extending $q\av_{\a(q)}$ if necessary, that there is some $a\in N$ such that $q\av_{\a(q)}$ forces that if $M\in\mtcl T_{\a+1}$ is such that $(M, \ssup(M\cap\a)))\in\D_p$ for some $p\in\dot G_\a$, $a\in M$, and $\d_M<\d_N$, then $\d_M\in\dot C^\a_\d$. We may of course assume that $$\d_0=\max(\{\d_Q\,:\,Q\in\mtcl N^q_{\a+1}\}\cap\d_N)$$ exists and $\d_0\in \dom(d^q_\a)$. We let $q^*$ be an extension of $q$ obtained by adding $a$ to $d^q_\a``\{\d_0\}$.

By our choice of $q$ and Lemma \ref{compl} together with the fact that $\a\in D$, we may assume, after extending $q\av_{\a(q)}$ if necessary, that there is some $a\in N$ such that $q\av_{\a(q)}$ forces that if $M\in\mtcl T_{\a+1}$ is such that $(M, \ssup(M\cap\a)))\in\D_p$ for some $p\in\dot G_\a$, $a\in M$, and $\d_M<\d_N$, then $\d_M\in\dot C^\a_\d$. We may of course assume that $\d_N\in \dom(d^q_\a)$. Let $\epsilon<\d_N$ be above $\d_M$ for every $M\in\mtcl N^q_{\a+1}$ such that $\d_M<\d_N$, and let  $q^*$ be an extension of $q$ obtained by adding $(\epsilon, a)$ to $d^q_\a``\{\d_N\}$ and closing under relevant isomorphisms $\Psi_{N_0, N_1}$.

We now show that $q^*$ and $\epsilon$ are as desired. For this, suppose $q'\in\mtcl P$ extends $q^*$ and $M\in\mtcl N^{q'}_{\a+1}$ is such that $\epsilon < \d_M<\d_N$. By an instance of clause (5)(b)(ii) in the definition of condition there is some $M'\in\mtcl N^{q'}_{\a+1}$ such that $\d_{M'}=\d_M$ and $a\in M'$. But then  $q'\av_{\a(q')}\Vdash_\a\d_{M'}=\d_M\in\dot C^\a_{\d_N}$ by the choice of $a$.
\end{proof}

%The following lemma will conclude the proof of Theorem \ref{mainthm}.

%\begin{lemma} $\mathcal P_\kappa$ forces $\mathfrak b(\omega_1)=\cf(\kappa)$.
%\end{lemma}

%\begin{proof}
%The argument for this is essentially contained in the proof of Lemma \ref{measuring}. Let $G$ be a $\mtcl P_\k$-generic filter. By the
%$\aleph_2$-c.c.\ of $\mtcl P_\kappa$, it is enough to show that, if $\alpha \in D$, then $D^G_\a$
%diagonalises all clubs of $\omega_1$ in $V[G_\a]$, where $G_\a=G\cap \mtcl P_\alpha$.

%So, let $\dot C\in H(\k)$ be a $\mtcl P_\a$-name for a club of $\omega_1$, let $q\in\mtcl P_\k$, and suppose $\dot C\in M$ for some $M\in\mtcl N^q_{\alpha+1}$. It suffices to show that if $q'\in\mtcl P_\k$ extends $q$ and $N\in\mtcl N^{q'}_{\alpha+1}$ is such that $\delta_M<\delta_{N}$, then $q'\Vdash_{\mtcl P_{\alpha+1}}\delta_{N}\in \dot C$. But by the fact that $q'\leq_\k q$ coupled with the shoulder axiom for $\mtcl N^{q'}_{a+1}$ we know that there is some $N'\in\mtcl N^{q'}_{\a+1}$ such that $\d_{N'}=\d_N$ and $\dot M\in N'$ for some $\dot M\in\mtcl N^{q'}_{\a+1}$ such that $\bar M\cong_{\a+1}M$. (and therefore $\dot M\in N'$).
%We then have that $q\arrowvert_\alpha$ is $(N', \mtcl P_\alpha)$-generic by Lemma \ref{properness2} and hence forces that $\delta_{N'}=\delta_N$ is a limit point of ordinals in $\dot C$ and therefore in $\dot C$.
%\end{proof}

\end{document}